\documentclass[10pt,twoside,a4paper]{article}
\usepackage[english,catalan,galician,activeacute]{babel}
\usepackage[T1]{fontenc}
\usepackage{inputenc}
\usepackage{amsmath,bm,mathrsfs}
\usepackage{amssymb}
\usepackage{amsxtra}
\usepackage{amsthm}
\usepackage{epic}
\usepackage{eepic} 
\usepackage{fancybox}
\usepackage{boxedminipage}
\usepackage{calc}
\usepackage{ifthen}
\usepackage{paralist}
\usepackage{exscale,relsize}
\usepackage{pgf}
\usepackage[rm,sc,small,center]{titlesec}
\titlelabel{\thetitle.\enspace}
\usepackage{tikz}
\usetikzlibrary{arrows,automata,backgrounds,calendar,snakes}
\usepackage{inputenc}
\usepackage{color}

\newcommand{\ol}[1]{ \overline{ #1 } }

\newcommand{\R}{\mathbb{R}}

\newcommand{\divv}{ \nabla \!\! \cdot \!}

\newcommand{\del}[2]{ \partial^{#1}_{#2} }
\newtheorem{theorem}{Theorem}[section]
\newtheorem{lemma}{Lemma}[section]

\newcommand{\rot}{} 
\begin{document}
\selectlanguage{english}
\thispagestyle{empty}
\numberwithin{equation}{section}
\pagestyle{myheadings}
\markboth{\small H. Freist\"uhler, M. Kotschote}{\small 
Flow of compressible two-phase fluids}

\title{The Blesgen and Lowengrub-Truskinovsky Descriptions 
of Two-Phase Compressible Fluid Flow:\\
Interstitial Working and a Reduction to\\ Korteweg Theory} 
\author{\it Heinrich Freist\"uhler\ ~ {\rm and }
Matthias Kotschote\thanks{both Department of Mathematics, 
University of Konstanz, 78457 Konstanz, Germany}}

\date{November 26, 2017}
\maketitle

\selectlanguage{english}

\section{Introduction}
Formulated by Blesgen \cite{B1} and Lowengrub and Truskinovsky \cite{LT}, 
the `Navier-Stokes-Allen-Cahn'  equations and the `Navier-Stokes-Cahn-Hilliard' 
equations describe diphasic fluid flow, combining 
the conservation laws for mass, momentum, and energy with a balance law for the
phases, which governs the concentration of one, and thus the other, phase  as an
order parameter. By contrast, the `Navier-Stokes-Korteweg' theory for the 
dynamics of a one-phase capillary fluid, given by Dunn and Serrin \cite{DS} following 
Korteweg  \cite{Korteweg} and Slemrod \cite{S}, 
uses only the three said conservation laws and has the mass density itself
as its order parameter.\footnote{We will henceforth abbreviate  
Navier-Stokes-Allen-Cahn, Navier-Stokes-Cahn-Hilliard, and Navier-Stokes-Korteweg 
as NSAC,NSCH, and NSK.} 

In \cite{FKARMA}, the authors considered the derivation of the
NSAC, the NSCH, and the NSK equations, and  showed that in the case that the 
two phase are incompressible with different specific volumes,
both NSAC and NSCH reduce to versions of NSK. However, as we realized only after 
that article was published, our assumptions in \cite{FKARMA} tacitly correspond to neglecting
the so-called microforces \cite{GPV}. While the argumentation in \cite{FKARMA} seems mathematically
consistent, the microforces are a physical reality, and the purpose of the present note 
is to show that the reduction of NSAC and NSCH to NSK still holds when one does take them 
into account.

Also in NSAC and NSCH, the microforces do what Dunn and Serrin named 
\emph{interstitial work}. In Section 2 we study the role that this particular effect 
plays in the entropy production along particle paths. Section 3 serves to  prove the reduction
property of NSAC and NSCH \emph{with}\ interstitial work. Section 4 
discusses the \emph{objectivity} of the three models briefly. In the sequel, we use
notation from \cite{FKARMA} and the results of Section 1 of that paper without defining /
stating them again.

\section{Interstitial work and entropy production} \label{sec:models}

\subsection{The Navier-Stokes-Korteweg system}
In this subsection we briefly recall entropy production in the Navier-Stokes-Korteweg 
system. This system, now exactly in its form derived by Dunn and Serrin in \cite{DS}, reads
\begin{equation}   \label{eq:NSK}
\begin{aligned}
\del{}{t} \rho + \divv( \rho u ) & = 0,
\\
\del{}{t} ( \rho u ) + \divv ( \rho u \otimes u  
- \bar{\mathbf{T}})& = 0,
\\
\del{}{t} \bar{\mathcal{E}} 
+ \divv ((\bar{\mathcal{E}}\mathbf{I}-\bar{\mathbf{T}}) u \rot{+\divv \bar w}) 
- \divv ( \beta \nabla \theta )& = 0,
\\
\end{aligned}
\end{equation}
with
\begin{equation*}
\bar{\mathcal{E}}\equiv \rho (\bar E+\frac12|u|^2) 
\quad\hbox{and}\quad 
\bar{\mathbf{T}}
\equiv
-p\mathbf{I}+\mathbf{K} + \mathbf{S}.  
\end{equation*}
The fluid is specified by its Helmholtz energy 
\begin{equation}\label{HelmholtzK}
\bar F(\rho,\theta,\nabla\rho)=\check{\bar F}(\rho,\theta,|\nabla\rho|^2/2),
\end{equation}
from which its internal energy $\bar E$ derives via the Legendre transform
$$
\bar E(\rho,s,\nabla\rho)\equiv \bar F(\rho,\theta,\nabla\rho)
+\theta s,
$$
with temperature $\theta$ and specific entropy 
$
s=-\partial_\theta F 
$
as dual variables.
The Korteweg tensor \rot{and the interstitial work flux} are given by 
\begin{equation}\label{eq:K}
\rot{\mathbf{K}
=
\bigg [\rho \divv \big( \partial_{\nabla\rho}({\rho} \bar{F} ) \big) \mathbf{I}
-
\nabla \rho \otimes \partial_{\nabla\rho }(\rho\bar F) \bigg],\quad
\bar w= \kappa\rho(\divv u)\nabla\rho.}
\end{equation}
with $\kappa$ given by
\begin{equation}
\rho\del {}{\nabla\rho}\bar F=\kappa\nabla\rho,
\end{equation}
and the reasoning in \cite{DS}\footnote{which can technically be expressed as a simple 
computation very similar 
to one in Sec.\ 2.3 of \cite{FKARMA}} yields the 
entropy balance 
\begin{equation}\label{eq:entropy}
\del{}{t} (\rho s) + \divv(\rho s u) = \divv \Sigma +\sigma
\end{equation}
with
\begin{equation}\label{eq:Sigma}
\Sigma\equiv
\beta\theta \nabla \theta
\end{equation}
and
\begin{equation}\label{eq:sigma}
\sigma 
\equiv
\frac{1}{\theta} \bigg( \eta(Du)^s:(Du)^s + \zeta (\divv u)^2 \bigg) 
+
\frac{\beta}{\theta^2} |\nabla\theta|^2.
\end{equation}
\subsection{The Navier-Stokes-phase-field models} \label{sec:nsac}
Before specializing to Navier-Stokes-Allen-Cahn and Navier-Stokes-Cahn-Hilliard,
we show some properties which both systems share.  We start from their common form
\begin{equation}   \label{phfm}
\begin{aligned}
\del{}{t} \rho + \divv( \rho u ) & = 0,
\\
\del{}{t} ( \rho u ) + \divv ( \rho u \otimes u  
- \mathbf{T})& = 0,
\\
\del{}{t} \mathcal{E} 
+ \divv ((\mathcal{E}\mathbf{I}-\mathbf{T}) u\rot{+w} ) 
- \divv ( \beta \nabla \theta )& = 0,
\\
\del{}{t} (\rho\chi) + \divv(\rho\chi u) 
-\rho j & = 0 
\end{aligned}
\end{equation}
in which
\begin{equation*}
\mathcal{E}\equiv \rho (E+\frac12|u|^2),
\quad 
\mathbf{T}
\equiv
-p\mathbf{I}+\mathbf{C} + \mathbf{S},
\quad \rot{w}  
\end{equation*}
are the total energy, the total Cauchy stress, \rot{and 
the interstitial work flux}.
We specify the fluid by its Gibbs energy $G$, i.e., the internal energy is
\begin{equation*}
E(\tau,s,\chi,\nabla\chi)\equiv G(p,\theta,\chi,\nabla\chi)-p\tau+\theta S,   
\end{equation*}
with temperature $\theta$ and specific entropy 
$
s=-\partial_\theta G 
$
as well as pressure $p$ and specific volume 
$
\tau=1/\rho=\partial_p G
$
as pairs of conjugate variables,
and the Ericksen tensor is 
\begin{equation*}\label{eq:Ericksentensor}
 \mathbf{C}
=-\rho\nabla\chi\otimes
\frac{\partial G}{\partial\nabla\chi}.
\end{equation*}
We will also use the generalized chemical potential 
\begin{align*}
\mu \equiv 
\del{}{\chi} G
- \frac1\rho\divv \bigg(\rho \del{}{\nabla \chi}  G \bigg),
\end{align*}

\begin{lemma}
For solutions to \eqref{phfm},
the entropy production along particle paths satisfies 
\begin{equation}\label{eq:dotsphasefield}
\rho\theta\dot s
=\mathbf{S}  : Du
+
\divv ( \beta \nabla \theta )
\rot{-\divv w}
 -\rho (G_{\nabla\chi} \cdot \nabla j + G_\chi j ).
\end{equation}
\end{lemma}
\begin{proof}
This is shown in exactly the same way as equation (2.7) in \cite{FKARMA}.
\end{proof}
\begin{lemma}
With 
\begin{equation}
\label{def:w}
w=-\rho j \del{}{\nabla \chi}G,
\end{equation}
equation \eqref{eq:dotsphasefield} reads
\begin{equation}
\rho\theta\dot s=
\mathbf{S}  : Du
+
\divv ( \beta \nabla \theta )
-\rho j\mu. 
\end{equation}
\end{lemma}
\begin{proof}
One computes
 $$
-\divv w
-
\rho (G_{\nabla\chi} \cdot \nabla j + G_\chi j )
=
j(-\rho G_\chi+\divv(\rho G_{\nabla\chi}))
$$
\end{proof}
The next result distinguishes between NSAC and NSCH.
\begin{lemma} Assume \eqref{def:w}.\\
(i) With 
\begin{equation}\label{jAC}
j=j_{AC}=-\frac\mu\epsilon,  
\end{equation}
the entropy production on solutions of \eqref{phfm} is given by \eqref{eq:entropy}
with
\begin{equation}
\Sigma=\divv \bigg(\frac\beta\theta \nabla \theta \bigg)
\end{equation}
and 
\begin{equation}\label{eq:sigmaAC}
\sigma=\frac1\theta\mathbf{S}  : Du
+
\frac\beta{\theta^2}|\nabla\theta|^2
+
\frac\epsilon\theta\rho j^2.
\end{equation}
(ii) With 
\begin{equation}\label{jCH}
j=j_{CH}=\frac1\rho\divv{\mathcal J},\quad {\mathcal J}=\gamma\nabla\bigg(\frac{\mu}{\theta}\bigg), 
\end{equation}
the entropy production  on solutions of \eqref{phfm} 
is given by \eqref{eq:entropy} with 
\begin{equation}\label{eq:SigmaCH}
 \Sigma
= 
\frac\beta\theta \nabla \theta 
-\gamma\frac\mu\theta\nabla\bigg(\frac\mu\theta\bigg)
\end{equation}
and  
\begin{equation}\label{eq:sigmaCH}
\sigma= 
\frac1\theta
\mathbf{S}  : Du
+
\frac\beta{\theta^2}|\nabla\theta|^2
+\gamma\bigg| \nabla\bigg(\frac\mu\theta\bigg)  \bigg|^2.
\end{equation}
\end{lemma}
\begin{proof} Immediate.
\end{proof}
Equations \eqref{phfm} with interstitial work flux $w$ as in \eqref{def:w}
and $j=j_{AC}, j=j_{CH}$ as in \eqref{jAC},\eqref{jCH} are the NSAC and NSCH equations
proposed in \cite{B1,LT}.\footnote{Though not to \emph{every} detail in the NSCH case, as  
Lowengrub and Truskinovsky seem to either use a law from isothermal diffusion 
or leave constitutive laws in particular for the heat and diffusion fluxes
widely open;
see equations (3.24b), (3.30) in \cite{LT}. The latter is  
very meaningful, cf.\ \cite{RT}.} 
We view the considerations of this subsection certainly not as a new derivation of NSAC 
and NSCH\footnote{Rather must we point out that Remarks 2.1, 2.2, 2.3 in \cite{FKARMA}
are no longer in order as soon as one does take the microforces into account.},
but as perhaps elucidating the role that interstitial work plays in connection with the second
law of thermodynamics also in these two phase-field models.

\section{Reduction to Korteweg models} \label{sec:reduction}
Assume now that a fluid consists of two incompressible phases of
different temperature-independent specific volumes. I.e. (cf.\ \cite{FKARMA}), its 
Gibbs energy has the form
\begin{equation}\label{R1}
G(p,\theta,\chi,\nabla\chi)=T(\chi)p+ W(\theta,\chi,\nabla\chi) 
\quad \hbox{with} \quad  
W(\theta,\chi,\nabla\chi)
=
\check W(\theta,\chi,|\nabla\chi|^2),
\end{equation}
where
\begin{equation}\label{R2}
 T(\chi)=\chi \tau_1+(1-\chi)\tau_2 
\end{equation}
with constants $\tau_1,\tau_2>0$ satisfying 
\begin{equation}\label{R3}
\tau_*\equiv\tau_1-\tau_2\neq 0. 
\end{equation}
\subsection{Reduction of the Navier-Stokes-Allen-Cahn system}
\begin{theorem} \label{thm:nsac}
In the case of two molecularly immiscible incompressible phases of
different, temperature-independent specific volumes, \eqref{R1}, \eqref{R2}, \eqref{R3},
the Navier-Stokes-Allen-Cahn equations \eqref{phfm} with 
$j=j_{AC}$ from \eqref{jAC} and $w$ from \eqref{def:w}
can be written as the Navier-Stokes-Korteweg system
\begin{equation} \label{eq:NSK-2}
\begin{split}
\del{}{t} \rho + \divv ( \rho u ) & = 0, 
\\
\del{}{t}  (\rho u)  + \divv ( \rho u \otimes u )
- 
\divv (-\bar p\mathbf{I}+\mathbf{K}+\mathbf{S}^\epsilon) & = 0,
\\
\del{}{t} \bar{\mathcal{E}} 
+ \divv (\bar{\mathcal{E}}u-
(-\bar p\mathbf{I}+\mathbf{K}+\mathbf{S}^\epsilon) 
u \rot{+\bar w}) 
- \divv ( \beta \nabla \theta )& = 0,
\end{split}
\end{equation}
with $\bar{\mathcal{E}},\bar p,\mathbf{K},\rot{\bar w}$ derived as in 2.3 from the Helmholtz energy 
\begin{equation*} 
\bar F(\theta,\rho,\nabla \rho) = W( \theta, \bar\chi(\rho), [\bar\chi'(\rho)]^2 |\nabla \rho|^2 )
\quad\hbox{with}\quad \bar\chi(\rho) := \frac{1/\rho - \tau_2}{ \tau_*} 
\end{equation*}
and with the modified viscous stress
\begin{equation}
\mathbf{S}^\epsilon = \eta (Du)^s
 + (\zeta+\zeta^\epsilon) \divv u \, \mathbf{I}
\quad\hbox{where}\quad
\zeta^\epsilon\equiv\frac{\epsilon}{\rho \tau_*^2}. \label{eq:S-2}
\end{equation}
\end{theorem}
\begin{proof} 
In this situation, the concentration $\chi$ and the total density $\rho$ are linked through 
 \begin{equation} \label{eq:hat-c}
\chi = \bar{\chi}(\rho), 
\end{equation}
and the bahvior of the concentration is described by
\begin{equation}\label{jisdivu}
\rho j=\del{}{t} (\rho \chi) + \divv ( \rho \chi u) = \rho \dot{\chi} = \rho \bar{\chi}'(\rho) \dot{\rho} 
= \frac{1}{\tau_*} \divv u. 
\end{equation}
Very similarly to \cite{FKARMA}, we have 
\begin{equation} 
\begin{split}
\del{}{\nabla \chi}G  & = \del{}{\nabla \chi}W 
= \frac{1}{ \bar\chi'(\rho) } \del{}{\nabla \rho}\bar{F}, 
\quad \nabla \chi = \bar\chi'(\rho) \nabla \rho.
\end{split}  
\end{equation}
Using \eqref{def:w} and \eqref{jisdivu}, this implies
\begin{align}
\label{wisbarw}
w&= -\rho \dot{\chi} \, \del{}{ \nabla \chi} G 
=  \rho^2 (\divv u) \, \bar{\chi}'(\rho) \, \del{}{ \nabla \chi} G
= \rho^2 (\divv u) \, \del{}{ \nabla \rho}\bar{F} 
= {\kappa}\rho (\divv u) \, \nabla\rho 
=\bar w. 
\end{align}
Furthermore, 
\begin{equation} 
\begin{split}
\del{}{\chi} G & = \tau_* p + \del{}{\chi} W
\end{split}  
\end{equation}
and
\begin{equation} \label{eq:12}
\frac{1}{\tau_*} \del{}{\chi}  W 
  = - \rho^2 \del{}{\rho} \bar F  
+ \rho^2 \frac{ \bar\chi''(\rho) }{ \bar\chi'(\rho) } \nabla \rho \cdot \del{}{\nabla \rho} \bar F,
\end{equation}
imply that 
\begin{align}
\label{mu}
\rho\mu&=-\bigg( \divv \bigg({\rho}\del{}{\nabla \chi}   G  \bigg) 
- {\rho}\del{}{\chi}  G  \bigg)
= 
-\bigg( \divv \bigg(  [ \bar\chi'(\rho) ]^{-1} \del{}{\nabla \rho} 
\bigg({\rho}\bar{F} \bigg) \bigg)  
- 
{\rho}\tau^* p 
- 
{\rho} \del{}{\chi} W  \bigg).
\end{align}
This allows to represent the pressure as
\begin{equation}\label{eq:p}
p = \frac{ \mu}{ \tau_*} \divv u 
- \rho \bar\chi'(\rho) \divv \left(  [ \bar\chi'(\rho) ]^{-1} \del{}{\nabla \rho} 
({\rho}\bar{F} ) \right) 
- \frac{1}{ \tau_*} \del{}{ \chi } W.
\end{equation}
Combining \eqref{eq:p} and \eqref{jisdivu} 
with the Allen-Cahn relation \eqref{jAC} 
gives
\begin{equation}\label{eq:pAC}
p =  - \frac{ \epsilon }{ \tau_*^2 \rho } \divv u 
- \rho \bar\chi'(\rho) \divv \left(  [ \bar\chi'(\rho) ]^{-1} \del{}{\nabla \rho} 
({\rho}\bar{F} ) \right) 
- \frac{1}{ \tau_*} \del{}{ \chi } W.
\end{equation}
Using this, we get
\begin{align*}
-p\mathbf{I}+\mathbf{C}+\mathbf{S} 
& = 
 - p \,  \mathbf{I} - \nabla \chi \otimes \del{}{\nabla \chi} ( \rho W)
+\mathbf{S}
\\
& = \frac{1}{ \tau_*} \del{}{\chi} W  \, \mathbf{I} 
+ 
\rho \bar\chi'(\rho) \divv \left(  [ \bar\chi'(\rho) ]^{-1} 
\del{}{\nabla \rho} ({\rho} \bar{F} ) \right) \, \mathbf{I} 
-
\nabla \rho \otimes \del{}{\nabla \rho} ( \rho \bar F )
+\mathbf{S}^\epsilon
\\
& = \frac{1}{ \tau_*} \del{}{\chi} W  \, \mathbf{I} 
+ 
\rho \divv \left(  \del{}{\nabla \rho} ({\rho}\bar{F} ) \right) \, \mathbf{I} 
-
\rho^2 \frac{ \bar\chi''(\rho) }{ \bar\chi'(\rho) } \nabla \rho \cdot \del{}{ \nabla \rho} \bar{F} 
-
\nabla \rho \otimes \del{}{\nabla \rho} ( \rho \bar F )+\mathbf{S}^\epsilon,
\end{align*}
where we have used \eqref{eq:hat-c} and \eqref{eq:S-2}. 
Replacing now the term $\del{}{ \chi } W$ according to the identity \eqref{eq:12},
we indeed find
\begin{equation}\label{CisK}
\begin{aligned}
-p\mathbf{I}+\mathbf{C}+\mathbf{S} 
& =
-\bar p \mathbf{I}  
+ 
\bigg[ \rho \divv \left(  \del{}{\nabla \rho} \bigg({\rho}\bar{F} \bigg) \right) \, \mathbf{I} 
- 
\nabla \rho \otimes \del{}{\nabla \rho} \bigg(\rho \bar F \bigg) \bigg] 
+\mathbf{S}^\epsilon
\\
& =
-\bar p\mathbf{I}+\mathbf{K}+\mathbf{S}^\epsilon
\end{aligned}
\end{equation}
with the Korteweg tensor $\mathbf{K}$ as defined in \eqref{eq:K}.
\end{proof}
As a consistency check, we note that in the situation of Theorem 3.1,
the entropy production rate 
\eqref{eq:sigmaAC} agrees with what \eqref{eq:sigma} becomes upon replacing 
$\zeta$ with $\zeta+\zeta^\epsilon$:  
\begin{equation*}
\frac{\epsilon\rho}{\theta}j_{AC}^2
= \frac{\epsilon}{ \rho \theta \tau_*^2} |\divv u|^2 
=\frac{\zeta^\epsilon}{\theta} |\divv u|^2. 
\end{equation*}


\subsection{Reduction of the Navier-Stokes-Cahn-Hilliard system}
\begin{theorem} \label{thm:nsch}
In the case of two molecularly immiscible incompressible phases of
different, temperature-independent specific volumes, \eqref{R1}, \eqref{R2}, \eqref{R3},
the Navier-Stokes-Allen-Cahn equations \eqref{phfm} with 
$j=j_{CH}$ from \eqref{jCH} and $w$ from \eqref{def:w}
can be written as the Navier-Stokes-Korteweg system
\begin{equation} \label{eq:NSK-3}
\begin{split}
\del{}{t} \rho + \divv ( \rho u ) & = 0, 
\\
\del{}{t}  (\rho u)  + \divv ( \rho u \otimes u )
- 
\divv (-\bar p\mathbf{I}+\mathbf{K}+\mathbf{S}_\gamma) & = 0,
\\
\del{}{t} \bar{\mathcal{E}} 
+ \divv (\bar{\mathcal{E}}u-
(-\bar p\mathbf{I}+\mathbf{K}+\mathbf{S}_\gamma) 
u \rot{+\bf w}) 
- \divv ( \beta \nabla \theta )& = 0,
\end{split}
\end{equation}
with $\bar{\mathcal{E}},\bar p,\mathbf{K},\rot{\bar w}$ derived as in 2.3 from the Helmholtz energy 
\begin{equation*}
\bar F(\theta,\rho,\nabla \rho) = W( \theta, \bar\chi(\rho), [\bar\chi'(\rho)]^2 |\nabla \rho|^2 )
\quad\hbox{with}\quad \bar\chi(\rho) := \frac{1/\rho - \tau_2}{ \tau_*} \label{eq:PSI-2}  
\end{equation*}
and with the modified viscous stress
\begin{equation} \label{eq:S-3}
\mathbf{S}_\gamma = 
\eta (Du)^s + \zeta\divv u \, \mathbf{I} 
+ 
\frac{\theta}{\tau^2_*} \Lambda_\gamma( \divv u) \mathbf{I}, 
\end{equation}
where $\Lambda_\gamma$ denotes the solution operator of the elliptic problem
\begin{equation*}
- \divv ( \gamma \nabla \phi) = \divv u \quad \hbox{on } \Omega.   
\end{equation*}
\end{theorem}
\begin{proof}
The only difference from the proof of Theorem \ref{thm:nsac}
lies in the form of the representation for the  pressure $p$.
We keep using identities \eqref{eq:hat-c} through \eqref{eq:p}.
As the Cahn-Hilliard relation \eqref{jCH} means that 
\begin{equation}
\label{muintermsofdivu}
\mu=-\frac\theta{\tau_*}\Lambda_\gamma(\divv u),
\end{equation}
identity \eqref{eq:p} now yields 
\begin{equation}\label{eq:pCH}
\begin{split}
p & =  
-\frac\theta{\tau_*^2}\Lambda_\gamma(\divv u)
- \rho \bar\chi'(\rho) \divv \left(  [ \bar\chi'(\rho) ]^{-1} \del{}{\nabla \rho} 
({\rho}\bar{F} ) \right) 
- \frac{1}{ \tau_*} \del{}{ \chi } W
\end{split}
\end{equation}
instead of \eqref{eq:p}. This leads to 
\begin{equation}\label{CisKgamma}
\begin{aligned}
-p\mathbf{I}+\mathbf{C}+\mathbf{S} 
& =
-\bar p\mathbf{I}+\mathbf{K}+\mathbf{S}_\gamma
\end{aligned}
\end{equation}
instead of \eqref{CisK}.
\end{proof}
For a confirmation, set
$$
\phi=\Lambda_\gamma(\divv u),
$$
note that for 
\eqref{eq:NSK-3}, the entropy balance
\eqref{eq:entropy} holds with
\begin{equation*}
\Sigma\equiv
\frac\beta\theta \nabla \theta-\frac\gamma{\tau_*^2}\phi\nabla\phi
\end{equation*}
and
\begin{equation*}
\sigma 
\equiv
\frac{1}{\theta} \bigg( \eta(Du)^s:(Du)^s + \zeta (\divv u)^2 \bigg) 
+
\frac{\beta}{\theta^2} |\nabla\theta|^2.
+
\frac\gamma{\tau_*^2}|\nabla\phi|^2
\end{equation*}
instead of \eqref{eq:Sigma} and  \eqref{eq:sigma}, 
and observe that in the present situation \eqref{muintermsofdivu} implies
that 
$$
\phi=-\tau_*\frac\mu\theta
$$
in accordance with \eqref{eq:SigmaCH}  and \eqref{eq:sigmaCH}.

\section{Spatiotemporal objectivity}
In \cite{FKARMA}, arguing against the interstitial work
that was so ingeniously discovered by Dunn and Serrin in \cite{DS},
we claimed that using the corresponding term in the energy equation  
``would violate the fundamental requirement that 
the contributions of a stress tensor $\mathbf{T}$
to momentum and energy balance are related as $\divv \mathbf{T}$ 
and $\divv (\mathbf{T}u)$'' (p.\ 10). However, 
instead of being fundamental, this `requirement' is fundamentally wrong --  in particular it 
is \emph{not} dictated by the principle of objectivity; see below.
Regarding NSK, besides by the persuasiveness of
Dunn and Serrin's argumentation based on the Clausius-Duhem inequality,   
we should have felt alarmed by the fact that  Benzoni-Gavage et al.\ in \cite{BDDJ}
had already identified a Hamiltonian structure for the Euler-Korteweg equations 
that confirmed Dunn and Serrin's formulation.
Regarding NSAC and NSCH, the equally solid derivations 
by Blesgen and Lowengrub and Truskinovsky were other signs of warning.

Still, we noticed the inappropriateness\footnote{at the level of physics --- 
it remains amazing to us that the NSAC-NSCH-NSK triple stays as 
consistent as shown in \cite{FKARMA} at the mathematical level, when one suppresses the interstitial work 
in all three models!} of the above `requirement' only 
when we began\footnote{partly paralleling work
on dissipation \cite{FT2}}
to formulate relativistic versions of NSK, NSAC, and NSCH. Also 3-tensors that reflect purely 
mechanical stresses must respect \emph{spatiotemporal objectivity}. 
In particular, the Korteweg tensor occuring in NSK and the phase transition rate / interphase
diffusion flux figuring as $j$ in NSAC and NSCH come from relativistically
covariant tensorial quantities that derive from energies which depend, objectively in the sense of Rational 
Mechanics, on the spatiotemporal 4-gradient of density/concentration, and interstitial work is a 
natural implication of this property. Related details can be 
found in \cite{F1,F2}.

\par\medskip



\footnotesize
\begin{thebibliography}{99}
\bibitem{BDDJ} 
\textsc{S. Benzoni-Gavage, R. Danchin, S.  Descombes, D. Jamet}: 
Structure of Korteweg models and stability of diffuse interfaces. 
\textit{Interfaces Free Bound.}\ \textbf{7} (2005), 371--414.

\bibitem{B1} 
\textsc{T. Blesgen}: A generalisation of the Navier-Stokes equations to
two-phase-flows. \textit{J. Phys. D: Appl. Phys.} \textbf{32} (1999), 1119--1123.
\bibitem{DS} 
\textsc{J. E. Dunn, J. Serrin}: On the thermomechanics of interstitial 
working. \textit{Arch.\ Rational Mech.\ Anal.} \textbf{88} (1985), 95--133.

\bibitem{F1} 
\textsc{H. Freist\"uhler}: 
A relativistic version of the Euler-Korteweg equations. Submitted.

\bibitem{F2} 
\textsc{H. Freist\"uhler}: 
Spatiotemporally objective descriptions of two-phase fluid flows.
Manuscript.
\bibitem{FKARMA}
\textsc{H. Freist\"uhler, M. Kotschote}: 
Phase-field and Korteweg-type models for the time-dependent flow 
of compressible two-phase fluids.
\textit{Arch.\ Rational Mech.\ Anal.} \textbf{224} (2017), 1-20.
\bibitem{FT2} 
\textsc{H. Freist\"uhler, B. Temple}: Causal dissipation for the relativistic dynamics of ideal gases. 
\textit{Proc.\ Roy.\ Soc.\ A} {\bf 473} (2017), 20160729, 20 pp.
\bibitem{GPV} \textsc{M. E. Gurtin, D. Polignone, J. Vi\~nals}:
\selectlanguage{english} 
Two-phase binary fluids and immiscible fluids described by an order parameter.
\textit{Math. Models Methods Appl. Sci.} \textbf{6} (1996), 815--831.
\bibitem{Korteweg}
\textsc{D.J. Korteweg}: Sur la forme que prennent les \'{e}quations
des mouvements des fluides si l'on tient compte des forces capillaires par des 
variations de densit\'{e}. 
\textit{Arch. N\'{e}er.\ Sci.\ Exactes S\'{e}r.\ II} {\bf 6} (1901), 1-24.
\bibitem{LT} 
\textsc{J. Lowengrub, L. Truskinovsky}:
Quasi-incompressible Cahn-Hilliard fluids and topological transitions.
\textit{Proc.\ R. Soc.\ Lond.\ A} \textbf{454} (1998), 2617--2654.
\bibitem{RT} 
\textsc{A. Roshchin, L. Truskinovsky}:
Model of a weakly non-local relaxing compressible medium.
\textit{J.\ Appl.\ Math.\ Mech.} \textbf{53} (1989), 904--910.

\bibitem{S} 
\textsc{M. Slemrod}:
Admissibility criteria for propagating phase boundaries in 
a van der Waals fluid. \textit{Arch.\ Rational Mech.\ Anal.} {\bf 81} (1983), 
301--315.
\end{thebibliography}
\end{document}